\documentclass[11pt,letterpaper]{article}
\usepackage[T1]{fontenc}
\usepackage[utf8]{inputenc}
\usepackage{lmodern}
\usepackage[margin=1in]{geometry}
\usepackage{amsmath,amssymb,amsthm,mathtools}
\usepackage{booktabs,array,tabularx}
\usepackage{microtype}
\usepackage{enumitem}
\usepackage{xurl}
\usepackage[hidelinks]{hyperref}
\hypersetup{pdftitle={A Counterexample to the Huneke--Wiegand Conjecture},
 pdfauthor={Son Pham},pdfsubject={Commutative algebra},
 pdfkeywords={Huneke-Wiegand conjecture, rigid ideal, numerical semigroup, Gorenstein domain}}
\newtheorem{theorem}{Theorem}[section]
\newtheorem{proposition}[theorem]{Proposition}
\newtheorem{lemma}[theorem]{Lemma}
\newtheorem{corollary}[theorem]{Corollary}
\theoremstyle{remark}
\newtheorem{remark}[theorem]{Remark}
\newcommand{\NN}{\mathbb N}
\newcommand{\ZZ}{\mathbb Z}
\newcommand{\QQ}{\mathbb Q}
\newcommand{\m}{\mathfrak m}
\newcommand{\Tors}{\operatorname{T}}
\newcommand{\tr}{\operatorname{tr}}
\newcommand{\End}{\operatorname{End}}
\newcommand{\Hom}{\operatorname{Hom}}
\newcommand{\Ext}{\operatorname{Ext}}
\newcommand{\Tor}{\operatorname{Tor}}
\newcommand{\Soc}{\operatorname{Soc}}
\newcommand{\PF}{\operatorname{PF}}
\newcommand{\edim}{\operatorname{edim}}
\newcommand{\length}{\ell}

\setlist[enumerate]{leftmargin=2em,itemsep=2pt,topsep=4pt}
\numberwithin{equation}{section}
\title{A Counterexample to the\ Huneke--Wiegand Conjecture}
\author{Son Pham\\[2pt]\small Independent researcher\\\small\href{mailto:sonphamorg@gmail.com}{\texttt{sonphamorg@gmail.com}}}
\date{6 September 2026}
\begin{document}
\maketitle
\begin{abstract}
We construct a symmetric numerical semigroup $\Gamma$ of multiplicity $56$, embedding dimension $26$, and Frobenius number $181$ such that, over every field $k$, the one-dimensional Gorenstein local domain $R=k[t^\Gamma]_{\m}$ has a nonprincipal ideal $I=(t^{56},t^{70})R$ for which $I\otimes_R\Hom_R(I,R)$ is torsion-free. This gives a counterexample to the Huneke--Wiegand conjecture in its general Gorenstein formulation. The proof reduces to an explicit finite sumset identity, interpreted through the Huneke--Iyengar--Wiegand colon criterion and Leamer's tensor-torsion formula. We also identify the tensor product with the trace ideal of $I$, compute $\End_R(I)=R[t^{101},t^{107}]$ and its conductor, and show that the counterexample persists under completion. A separate graded tensor-graph calculation and exact, standard-library Python verifiers accompany the paper. We explain how the example differs from the field-extension construction of Christensen, Gerko, and Iyengar.
\end{abstract}
\noindent\small\textbf{2020 Mathematics Subject Classification.} Primary 13D07; Secondary 13C12, 13H10, 20M14.\\
\textbf{Keywords.} Huneke--Wiegand conjecture, Gorenstein local domain, rigid ideal, tensor torsion, numerical semigroup, trace ideal.\normalsize

\section{Introduction}
The Huneke--Wiegand conjecture, in its one-dimensional Gorenstein formulation, asserts that a finite torsion-free module $M$ over a Gorenstein local domain $R$ is free whenever $M\otimes_R M^*$ is torsion-free, where $M^*=\Hom_R(M,R)$; see \cite{HW,HWcorr,HIW}. We give an explicit counterexample with $M$ a two-generated monomial ideal. All computations are independent of the characteristic of the ground field.

Throughout, $\NN=\{0,1,2,\ldots\}$, and $[p,q]$ denotes the set of integers between $p$ and $q$, inclusive. Sums and translates of subsets of $\ZZ$ are setwise. Define
\begin{equation}\label{eq:G}
\begin{split}
G=\{&56,57,58,63,64,70,71,72,73,74,75,76,77,78,79,80,81,82,83,\\
&87,89,90,93,95,96,97\},\qquad \Gamma=\langle G\rangle.
\end{split}
\end{equation}

\begin{theorem}\label{thm:main}
Let $k$ be any field, and set
\[
 A=k[t^\Gamma],\qquad \m=(t^n:n\in\Gamma,\ n>0),\qquad R=A_{\m}.
\]
Then $R$ is a one-dimensional Gorenstein local domain. The ideal
\[
 I=(t^{56},t^{70})R
\]
is a nonfree torsion-free module of rank one, while $I\otimes_R\Hom_R(I,R)$ is torsion-free. In particular, the general Gorenstein formulation of the Huneke--Wiegand conjecture is false.
\end{theorem}

The finite core of the proof is the identity
\begin{equation}\label{eq:core}
 B+B=C,
\end{equation}
where
\begin{equation}\label{eq:BC}
\begin{split}
 B&=\{56,57,58,63,64,73,75,76,79,81,82,83\},\\
 C&=[112,116]\cup[119,122]\cup[126,152]\cup[154,166].
\end{split}
\end{equation}
Sections~\ref{sec:colon} and~\ref{sec:inverse} interpret this identity in two algebraic ways. Section~\ref{sec:structure} identifies the tensor product with
\[
 T=(t^g:g\in G\setminus\{74,80\})R
\]
and proves
\[
 \length_R(R/T)=3,\qquad
 \End_R(I)=R[t^{101},t^{107}],\qquad
 (R:\End_R(I))=T.
\]
We also record the completed example and a $\Tor_1$-vanishing consequence. Section~\ref{sec:graphs} gives a separate verification directly from tensor-product balancing relations.

Christensen, Gerko, and Iyengar~\cite{CGI} give a different construction using linear algebra in a finite field extension; their paper also records the present example and Huneke's independent verification. A comparison appears in Section~\ref{sec:comparison}. The purpose here is to place the explicit semigroup construction, its proofs, and its structural invariants on record. No minimality claim is required or made for our proof.

\section{The semigroup, the local ring, and the ideal}\label{sec:ring}
\begin{lemma}\label{lem:membership}
The elements of $\Gamma$ through $181$ are
\begin{equation}\label{eq:membership}
\begin{split}
\Gamma\cap[0,181]=\{0\}&\cup[56,58]\cup[63,64]\cup[70,83]\cup\{87\}\\
&\cup[89,90]\cup\{93\}\cup[95,97]\cup[112,116]\\
&\cup[119,122]\cup[126,180].
\end{split}
\end{equation}
Every integer at least $182$ belongs to $\Gamma$. Thus the Frobenius number is $181$, the conductor is $182$, and
\begin{equation}\label{eq:gaps}
\begin{split}
\NN\setminus\Gamma={}&[1,55]\cup[59,62]\cup[65,69]\cup[84,86]\cup\{88\}\\
&\cup[91,92]\cup\{94\}\cup[98,111]\cup[117,118]\\
&\cup[123,125]\cup\{181\}.
\end{split}
\end{equation}
\end{lemma}
\begin{proof}
Every positive element below $112=2\cdot56$ must be in $G$. Direct addition gives
\[
 (G+G)\cap[0,181]=[112,116]\cup[119,122]\cup[126,180].
\]
A sum of three positive generators that is at most $181$ cannot use a generator at least $70$, since $70+56+56=182$. It must therefore use three elements of $\{56,57,58,63,64\}$. All such sums below $181$ lie in $[168,180]$, already present above. The value $181$ would require three elements of $\{0,1,2,7,8\}$ to sum to $13$, which is impossible. Four positive summands have sum at least $224$. This proves~\eqref{eq:membership}.

Next,
\[
 [182,236]=56+[126,180],\qquad 237=57+180.
\]
Hence $56$ consecutive integers starting at $182$ belong to $\Gamma$. Addition by $56\in\Gamma$ propagates this interval to all larger integers. Since $181\notin\Gamma$, the remaining assertions follow.
\end{proof}

\begin{proposition}\label{prop:symmetry}
The numerical semigroup $\Gamma$ is symmetric. It has multiplicity $56$, embedding dimension $26$, and genus $91$.
\end{proposition}
\begin{proof}
Reflection of the interval lists in~\eqref{eq:membership} and~\eqref{eq:gaps} gives
\[
 n\in\Gamma\quad\Longleftrightarrow\quad181-n\notin\Gamma
 \qquad(0\leq n\leq181).
\]
This is symmetry. The least positive element is $56$. Every member of $G$ is less than $112$, so none is a sum of two positive elements of $\Gamma$; thus the $26$ displayed generators are minimal. Each of the $91$ pairs $\{n,181-n\}$ contains exactly one gap, giving genus $91$.
\end{proof}

The ring $A$ is a finite integral extension of $k[t^{56}]$, since each generator $t^g$ satisfies the monic equation $X^{56}-(t^{56})^g=0$. It is therefore a one-dimensional Noetherian domain. Also, $t=t^{57}/t^{56}$, so its fraction field is $K=k(t)$. The ideal $\m=A_+$ is maximal with residue field $k$, and $R=A_\m$ is a one-dimensional local domain. Write $v$ for the $t$-adic valuation on $K$.

\begin{lemma}\label{lem:valuation}
The values of the nonzero elements of $R$ are exactly $\Gamma$.
\end{lemma}
\begin{proof}
Write an element of $R$ as $f/g$, where $f,g\in A$ and $g\notin\m$. The denominator has nonzero constant term, hence $v(g)=0$ and $v(f/g)=v(f)\in\Gamma$. Conversely, $v(t^n)=n$ for every $n\in\Gamma$.
\end{proof}

\begin{proposition}\label{prop:Gorenstein}
The ring $R$ is Gorenstein, with $\edim R=26$ and Hilbert--Samuel multiplicity $e(R)=56$.
\end{proposition}
\begin{proof}
The Gorenstein assertion follows from symmetry and the symmetry criterion for numerical semigroup rings~\cite{Kunz}. Here is also a direct socle check. A gap $q$ is pseudo-Frobenius if $q+g\in\Gamma$ for every positive $g\in\Gamma$. If $q<181$ is a gap, symmetry gives $181-q\in\Gamma\setminus\{0\}$, so $q$ is not pseudo-Frobenius. The gap $181$ is pseudo-Frobenius by the conductor. Thus $\PF(\Gamma)=\{181\}$.

Let $a=t^{56}$. Monomial membership in $((a):_A\m)$ is equivalent to
\[
 x\in\Gamma,\qquad x+g-56\in\Gamma\quad\text{for every }g\in G.
\]
For $x-56\notin\Gamma$, the only nonnegative possibility is $x-56=181$. The only possible negative case is $x=0$, which fails for $g=57$. It follows, after localization, that
\[
 ((a):_R\m)=(t^{56},t^{237})R,\qquad
 \Soc(R/(a))=k\,\overline{t^{237}}.
\]
A one-dimensional local domain is Cohen--Macaulay, and $(a)$ is a parameter ideal. The one-dimensional socle therefore proves that $R$ is Gorenstein by the socle criterion~\cite[Theorem~18.1]{Matsumura}.

The minimal monomial generators of $\m$ correspond to $G$, so $\edim R=26$. The interval calculation in Lemma~\ref{lem:membership} gives $[224,279]\subseteq(G+G)+(G+G)$. Addition by $56$ shows that $\m^4$ has exactly the monomial exponents at least $224$, and then $\m^5=t^{56}\m^4$. Thus $(t^{56})$ is a reduction of $\m$. The quotient $R/(t^{56})$ has one Ap\'ery monomial for each residue class modulo $56$, and hence length $56$. The multiplicity formula for a parameter reduction in a Cohen--Macaulay local ring gives $e(R)=56$.
\end{proof}

\begin{proposition}\label{prop:nonprincipal}
The ideal $I=(t^{56},t^{70})R$ is nonprincipal. It is torsion-free of rank one and has exactly two minimal generators.
\end{proposition}
\begin{proof}
Its value set is
\[
 v(I)=(56+\Gamma)\cup(70+\Gamma).
\]
Indeed, monomial membership before localization is termwise, and denominators used in localization have value zero. If $I=fR$, its least value forces $v(f)=56$. Since $t^{70}\in fR$, some $r\in R$ would satisfy $t^{70}=fr$, and then $v(r)=14\notin\Gamma$, contrary to Lemma~\ref{lem:valuation}. A rank-one free ideal over a local ring is principal. The remaining assertions follow since $I$ is a nonzero ideal in a domain.
\end{proof}

\section{Colon ideals and rigidity}\label{sec:colon}
Put $a=t^{56}$ and $b=t^{70}$. First work over $A$, and set
\[
 P_A=((a)A:_A b),\qquad Q_A=((b)A:_A a),\qquad
 E=\{x\in\ZZ:x,x+14\in\Gamma\}.
\]
For a nonempty subset $U\subseteq\ZZ$ bounded below, write $(t^U)A$ for the monomial fractional ideal generated by the $t^u$ with $u\in U$.

\begin{lemma}\label{lem:colons}
One has $P_A=(t^E)A$ and $Q_A=(t^{14+E})A$.
\end{lemma}
\begin{proof}
A monomial $t^x\in A$ belongs to $P_A$ exactly when $t^{x+70}\in t^{56}A$, or $x+14\in\Gamma$. Polynomial membership is termwise because the monomials $t^n$, $n\in\Gamma$, form a $k$-basis of $A$. The second equality follows in the same way from the condition $x-14\in E$.
\end{proof}

Retain $B,C$ from~\eqref{eq:BC}, and put
\[
 X=14+C=[126,130]\cup[133,136]\cup[140,166]\cup[168,180].
\]
\begin{lemma}[Finite colon certificate]\label{lem:certificate}
The following equalities of subsets of $\ZZ$ hold:
\begin{align}
 E&=B+\Gamma,\label{eq:E}\\
 B+B&=C,\label{eq:BB}\\
 E\cap(14+E)&=X+\Gamma.\label{eq:intersection}
\end{align}
\end{lemma}
\begin{proof}
Lemma~\ref{lem:membership} and the conductor give
\begin{equation}\label{eq:Einterval}
\begin{split}
E=B&\cup[112,116]\cup[119,122]\cup[126,166]\\
 &\cup[168,180]\cup[182,\infty).
\end{split}
\end{equation}
Direct addition gives
\[
 B+G=[112,116]\cup[119,122]\cup[126,166]\cup[168,180].
\]
The tail lies in $B+\Gamma$ because
\begin{align*}
 [182,222]&=56+[126,166],&223&=57+166,\\
 [224,236]&=56+[168,180],&237&=57+180,
\end{align*}
and addition by $56$ propagates the block $[182,237]$. This proves~\eqref{eq:E}; the reverse inclusion also follows at once because $E$ is closed under addition by $\Gamma$.

The equality~\eqref{eq:BB} is a finite addition table. Appendix~\ref{app:pairs} lists a representation of every element of $C$; enumerating all unordered pairs from $B$ produces no value outside $C$.

Finally, intersecting~\eqref{eq:Einterval} with its translate gives
\[
 E\cap(14+E)=X\cup[182,194]\cup[196,\infty).
\]
A further finite addition yields
\[
 X+G=[182,194]\cup[196,277].
\]
The last interval has length greater than $56$, so addition by $56$ supplies its entire tail. This proves~\eqref{eq:intersection}.
\end{proof}

\begin{proposition}\label{prop:colonidentity}
One has $P_A\cap Q_A=P_AQ_A$. Consequently, for $P=P_AR$ and $Q=Q_AR$,
\begin{equation}\label{eq:colonidentity}
 P\cap Q=PQ.
\end{equation}
\end{proposition}
\begin{proof}
The ideals over $A$ are monomial. By Lemma~\ref{lem:certificate}, the intersection has exponent set $X+\Gamma$, and the product has exponent set
\[
 E+(14+E)=14+(B+B)+\Gamma=X+\Gamma.
\]
Thus they are equal.

Colons by finitely generated ideals commute with localization: finitely many denominators can be cleared from a finite generating set. Flatness of localization also preserves finite intersections, and localization preserves products. The asserted equality over $R$ follows. In particular,
\[
 P=((a)R:_R b),\qquad Q=((b)R:_R a).
\]
\end{proof}

\begin{proof}[Proof of Theorem~\ref{thm:main}]
The ring and nonprincipality assertions were proved in Section~\ref{sec:ring}. Huneke--Iyengar--Wiegand's two-generator criterion~\cite[Corollary~5.7]{HIW} says that, over a one-dimensional Gorenstein local ring, an ideal $(a,b)$ with $a,b$ regular is rigid precisely when~\eqref{eq:colonidentity} holds. Here and below the numbered results cited from~\cite{HIW} refer to its arXiv version~1. Since $R$ is a domain, both generators are regular. Therefore
\[
 \Ext_R^1(I,I)=0.
\]
For a finite torsion-free module of rank over a one-dimensional Gorenstein local ring, rigidity is equivalent to torsion-freeness of the tensor product with its dual~\cite[Proposition~5.1]{HIW}. This proves the theorem.
\end{proof}

\section{An inverse-ideal proof of torsion-freeness}\label{sec:inverse}
Normalize inside $K$ by
\[
 J=t^{-56}I=(1,t^{14})R\cong I.
\]
The following normalized form of Leamer's two-generator formula~\cite[Lemma~1.9]{Leamer} gives a second proof. We include its short proof to make this route self-contained.

\begin{lemma}\label{lem:Leamer}
Let $D_0$ be a domain with fraction field $K_0$, let $z\in K_0$, and set $L=(1,z)D_0$ and $H=(D_0:_{K_0}L)$. The torsion submodule of $L\otimes_{D_0}H$ is isomorphic to
\[
 \frac{(D_0:_{K_0}L^2)}{H^2}.
\]
\end{lemma}
\begin{proof}
Every $D_0$-linear map $L\to D_0$ extends over $K_0$ to multiplication by a scalar; hence $H$ identifies naturally with $L^*$. There is an exact sequence
\[
 0\longrightarrow H\xrightarrow{\;h\mapsto(-zh,h)\;}D_0^2
 \xrightarrow{\;(u,v)\mapsto u+zv\;}L\longrightarrow0.
\]
By right exactness, tensoring with $H$ presents $L\otimes H$ as the quotient of $H\oplus H$ by the image of $H\otimes H$. That image is $\{(-zw,w):w\in H^2\}$, where $H^2$ is the product ideal, not the tensor square. No injectivity after tensoring is asserted. The multiplication map from this quotient to $LH\subset K_0$ has kernel represented by pairs $(-zv,v)$ with $v\in H$ and $zv\in H$. These conditions are exactly $v,zv,z^2v\in D_0$, that is, $v\in(D_0:L^2)$.

After passing to $K_0$, multiplication is an isomorphism, so its kernel is torsion. Its target is torsion-free, so every torsion element lies in the kernel. This identifies the torsion submodule with the stated quotient.
\end{proof}

Over $A$, set $J_A=(1,t^{14})A$ and $L^{-1}=(A:_K L)$ for a nonzero fractional ideal $L$. Since $1\in J_A$, the inverse ideal is contained in $A$ and its membership condition is termwise. Hence
\[
 v(J_A^{-1})=E,\qquad
 v((J_A^2)^{-1})=D:=\{x:x,x+14,x+28\in\Gamma\}.
\]
By definition, $14+D=E\cap(14+E)$. Lemma~\ref{lem:certificate} therefore gives
\[
 D=C+\Gamma=(B+B)+\Gamma=E+E.
\]
Equality of these monomial exponent sets proves
\begin{equation}\label{eq:inverseidentity}
 (J_A^{-1})^2=(J_A^2)^{-1}.
\end{equation}
Localization gives $(J^{-1})^2=(J^2)^{-1}$. Applying Lemma~\ref{lem:Leamer} to $J$ yields
\[
 \Tors(J\otimes_R J^*)=0.
\]
Since $I\cong J$, their tensor products with their duals are isomorphic. This proves the torsion-free assertion without using the rigidity criterion.

\section{The trace ideal and the endomorphism ring}\label{sec:structure}
The trace of a finite $R$-module $M$ is the image of the evaluation map
\[
 M\otimes_R M^*\longrightarrow R,\qquad x\otimes f\longmapsto f(x).
\]
For a nonzero fractional ideal $L\subset K$, we identify $L^*$ with $(R:_K L)$; its trace is then $LL^{-1}$. The calculations below are first made for monomial fractional ideals over $A$, where membership is termwise, and then localized. Thus they identify actual ideals and modules, not merely value sets.

\begin{proposition}\label{prop:trace}
Put
\[
 T=(t^g:g\in G\setminus\{74,80\})R.
\]
The evaluation map induces an isomorphism
\begin{equation}\label{eq:traceiso}
 I\otimes_R I^*\xrightarrow{\ \sim\ }\tr_R(I)=T.
\end{equation}
Moreover,
\[
 \length_R(R/T)=3,\quad \length_R(\m/T)=2,\quad
 \mu_R(I^*)=12,\quad\mu_R(T)=24.
\]
Here $\mu_R$ denotes the minimal number of generators. A $k$-basis of $R/T$ is $1,\overline{t^{74}},\overline{t^{80}}$.
\end{proposition}
\begin{proof}
Trace is invariant under module isomorphism, so we may replace $I$ by $J=(1,t^{14})R$. Its inverse has exponent set $E=B+\Gamma$, and therefore
\[
 \tr_R(I)=JJ^{-1}=J^{-1}+t^{14}J^{-1}.
\]
The finite identity
\[
 B\cup(14+B)=G\setminus\{74,80\}
\]
identifies this product with $T$. Taking the union of~\eqref{eq:Einterval} with its translate gives the more precise equality
\begin{equation}\label{eq:tracevalues}
 E\cup(14+E)=\Gamma\setminus\{0,74,80\}.
\end{equation}
The corresponding quotient over $A$ has precisely the three asserted monomial basis elements. It is supported at $\m$, so localization preserves its $k$-dimension and gives the length statements.

The $12$ generators $t^b$, $b\in B$, of $J^{-1}$ are minimal: all exponents are below $112$, whereas every exponent in $\m J^{-1}$ is at least $112$. Scaling identifies $I^*$ with $J^*$. The same argument proves minimality of the $24$ displayed generators of $T$.

Finally, the evaluation map is multiplication under the fractional-ideal identification. It becomes an isomorphism over $K$, so its kernel is torsion. Theorem~\ref{thm:main} makes the source torsion-free; hence this kernel vanishes. Its image is the trace ideal, proving~\eqref{eq:traceiso}.
\end{proof}

\begin{proposition}\label{prop:end}
Identify $S=\End_R(I)$ with $(I:_K I)$. Then
\begin{align}
 S&=R+Rt^{101}+Rt^{107}=R[t^{101},t^{107}],\label{eq:end}\\
 v(S)&=\Gamma\cup\{101,107,181\},\label{eq:endvalues}\\
 (R:_K S)&=T.\label{eq:conductor}
\end{align}
In particular,
\[
 S/R=k\,\overline{t^{101}}\oplus k\,\overline{t^{107}}
      \oplus k\,\overline{t^{181}},\qquad
 \length_R(S/R)=3=\length_R(R/T).
\]
\end{proposition}
\begin{proof}
Scaling does not change the endomorphism ring, so $S=(J:_K J)$. Over $A$, put $U=\Gamma\cup(14+\Gamma)$, the exponent set of $J_A$. Since $1\in J_A$, a multiplier belongs to $(J_A:J_A)$ exactly when its monomials have exponents in
\[
 U\cap(-14+U).
\]
For completeness, Lemma~\ref{lem:membership} gives
\[
 U\setminus\Gamma=
 \{14,84,85,86,88,91,92,94,101,103,104,107,109,110,111,181\}.
\]
Among these exponents $n$, precisely $101,107,181$ also satisfy $n+14\in U$. All exponents in $\Gamma$ satisfy that condition. Moreover,
\[
 181=101+80=107+74,
\]
and every other positive translate of $101$ or $107$ by $\Gamma$ is already in $\Gamma$. Thus the endomorphism ring is generated over $R$ by $1,t^{101},t^{107}$. Products of the two additional generators have exponents at least $202$ and hence lie in $R$. This proves~\eqref{eq:end} and~\eqref{eq:endvalues}, as well as the quotient basis and its length.

Using these $R$-module generators, the conductor has exponent set
\[
 \{n\in\Gamma:n+101,n+107\in\Gamma\}.
\]
The value $n=0$ fails. For positive $n\in\Gamma$, one has $n\geq56$, so the only possible gap encountered by either translate is $181$. These failures occur exactly at $n=80$ and $n=74$, respectively. The conductor therefore has exponent set $\Gamma\setminus\{0,74,80\}$, which is the exponent set of $T$ in~\eqref{eq:tracevalues}. Termwise membership and localization prove~\eqref{eq:conductor}.
\end{proof}

\begin{remark}
The last length equality is consistent with one-dimensional Gorenstein duality for a finite birational extension. The contribution of Proposition~\ref{prop:end} is the explicit description of this extension and its conductor; no new general trace--conductor or duality principle is being asserted.
\end{remark}

\section{Completion and a homological consequence}\label{sec:consequences}
\begin{corollary}\label{cor:completion}
For every field $k$, let
\[
 \widehat R=k[[t^\Gamma]],\qquad
 \widehat I=(t^{56},t^{70})\widehat R.
\]
Then $\widehat R$ is a one-dimensional complete Gorenstein local domain, $\widehat I$ is nonprincipal, and
\[
 \widehat I\otimes_{\widehat R}
 \Hom_{\widehat R}(\widehat I,\widehat R)
 \cong T\widehat R
\]
is torsion-free. Thus the counterexample persists under completion.
\end{corollary}
\begin{proof}
The completion of $R$ is the displayed numerical semigroup power-series ring; it embeds in $k[[t]]$ and is a domain. Gorensteinness is preserved by completion. The completion map is faithfully flat; see the standard completion results in~\cite{Matsumura}. Since $I$ is finitely presented, flat base change gives
\[
 \Hom_R(I,R)\otimes_R\widehat R
 \cong\Hom_{\widehat R}(I\otimes_R\widehat R,\widehat R).
\]
Indeed, a finite presentation expresses each side as the same kernel between finite free $\widehat R$-modules. Tensoring~\eqref{eq:traceiso} with $\widehat R$ therefore gives the claimed isomorphism. Flatness identifies $T\otimes_R\widehat R$ with the ideal $T\widehat R$, which is torsion-free because $\widehat R$ is a domain. Finally,
\[
 \dim_k(\widehat I/\widehat\m\widehat I)
 =\dim_k(I/\m I)=2,
\]
so $\widehat I$ remains nonprincipal.
\end{proof}

\begin{corollary}\label{cor:Tor}
For $P=((t^{56}):_R t^{70})$ and $Q=((t^{70}):_R t^{56})$, one has
\[
 \Tor_1^R(R/P,R/Q)=0,\qquad
 \length_R(R/P)=16,\qquad\length_R(R/Q)=30.
\]
Both quotients are nonzero Artinian Gorenstein rings.
\end{corollary}
\begin{proof}
Tensor the sequence $0\to P\to R\to R/P\to0$ with $R/Q$. The resulting kernel identifies
\[
 \Tor_1^R(R/P,R/Q)\cong(P\cap Q)/PQ=0
\]
by Proposition~\ref{prop:colonidentity}. The complement in $\Gamma$ of the exponent set $E$ of $P$ is
\[
 \{0,70,71,72,74,77,78,80,87,89,90,93,95,96,97,167\}.
\]
For $Q$, whose exponent set is $14+E$, the complement is
\[
 \begin{split}
 \{0,56,57,58,63,64,73,74,75,76,79,80,81,82,83\}
 &\cup[112,116]\cup[119,122]\\[-2pt]
 &\cup\{131,132,137,138,139,195\}.
 \end{split}
\]
These give the lengths. A monomial class is in the socle precisely when adding every $g\in G$ to its exponent puts it into the relevant ideal. The two finite lists show that the socles are, respectively,
\[
 k\,\overline{t^{167}}\quad\text{and}\quad k\,\overline{t^{195}}.
\]
Termwise membership also rules out additional socle classes from linear combinations. The one-dimensional socle criterion proves the last assertion.
\end{proof}

Multiplication by $t^{14}$ gives an $R$-module isomorphism $P\cong Q$, since $Q=t^{14}P$. Corollary~\ref{cor:Tor} thus realizes the configuration discussed in Huneke--Iyengar--Wiegand~\cite[Remark~5.8]{HIW}: isomorphic proper ideals whose quotients are Gorenstein and whose first Tor vanishes.

\section{Direct graded tensor verification}\label{sec:graphs}
We give a separate verification that starts from balancing relations and does not use the rigidity criterion or the inverse-ideal quotient formula. This is a monomial tensor-graph calculation of the kind considered in~\cite{Leamer}. It is not needed for the algebraic proof above.

Set $I_A=(t^{56},t^{70})A$, and identify
\[
 I_A^*=\Hom_A(I_A,A)=(A:_K I_A).
\]
Let $U=v(I_A)$ and $V=v(I_A^*)$. For each integer $d$, let $\mathcal G_d$ be the finite graph with vertices
\[
 (u,v)\in U\times V,\qquad u+v=d.
\]
Two distinct vertices are joined when, after interchanging them if necessary,
\[
 u_1-u_2=v_2-v_1\in\Gamma.
\]

\begin{lemma}\label{lem:graph}
The degree-$d$ part of $I_A\otimes_A I_A^*$ has one $k$-basis vector for each connected component of $\mathcal G_d$. Under multiplication
\[
 I_A\otimes_A I_A^*\longrightarrow I_AI_A^*,
\]
all these basis vectors map to the same monomial $t^d$. Consequently, the degree-$d$ part of the torsion submodule has dimension
\[
 \max\{0,\#\pi_0(\mathcal G_d)-1\}.
\]
\end{lemma}
\begin{proof}
The monomials with exponents in $U$ and $V$ are bases of the two fractional ideals. In degree $d$, tensoring over $A$ imposes exactly the relations
\[
 t^{u+g}\otimes t^v=t^u\otimes t^{v+g}\qquad(g\in\Gamma).
\]
These identify vertices precisely within connected components. The resulting quotient has one basis vector per component over every field, and multiplication sends each to $t^d$. Its kernel is exactly torsion: after extending scalars to $K$ the map is an isomorphism, and its target is a submodule of $K$ and hence torsion-free.
\end{proof}

The calculation is finite for a proved reason, not an experimental cutoff. Here
\[
 a_0=\min U=56,\qquad b_0=\min V=0,\qquad c=182.
\]
Indeed $0\in V$. If $x<0$ belonged to $V$, then $x+56\in\Gamma\cap[0,55]=\{0\}$, so $x=-56$, but then $x+70=14\notin\Gamma$.

\begin{lemma}[Uniform tail bound]\label{lem:tail}
Every nonempty graph $\mathcal G_d$ is connected for
\[
 d\geq a_0+b_0+4c=784.
\]
\end{lemma}
\begin{proof}
The sets $U,V$ are relative $\Gamma$-ideals. Every integer at least $a_0+c$ belongs to $U$, and every integer at least $b_0+c$ belongs to $V$. For $d\geq a_0+b_0+4c$, use the vertices
\[
 H_1=(a_0+c,d-a_0-c),\qquad H_2=(d-b_0-c,b_0+c).
\]
For any vertex $(u,v)$ of degree $d$, either $u\geq a_0+2c$ or $v\geq b_0+2c$. In the first case $u-(a_0+c)\geq c$ is in $\Gamma$, giving an edge to $H_1$; in the second case there is an edge to $H_2$. The difference between the first coordinates of the hubs is
\[
 d-a_0-b_0-2c\geq2c\geq c,
\]
so it also belongs to $\Gamma$, and the hubs are joined. Thus the graph is connected.
\end{proof}

The ancillary program \texttt{verify\_tensor\_graph.py} computes every graph in the remaining range $56\leq d<784$ using independently generated, exact semigroup membership data. It finds that every nonempty graph is connected. Lemmas~\ref{lem:graph} and~\ref{lem:tail} give a direct computer-assisted verification over $A$, and localization recovers the assertion over $R$. No matrix-rank calculation or characteristic-dependent operation is used.

\section{Comparison with related work}\label{sec:comparison}
Christensen, Gerko, and Iyengar~\cite[Propositions~2.2 and~3.2]{CGI} construct rigid ideals using subspaces of a finite field extension. Their Example~3.3 uses a degree-$24$ extension of $\QQ$ and a standard graded domain with an ideal generated by two linear forms. Both constructions use the colon criterion from~\cite{HIW}; their underlying rings and finite certificates differ.

Table~\ref{tab:comparison} compares the localizations at the homogeneous maximal ideals. The invariants in its middle column were proved above. For the last column, the graded dimensions in~\cite[Example~3.3]{CGI} are $1,12,23,24,24,\ldots$, yielding Hilbert series
\[
 \frac{1+11z+11z^2+z^3}{1-z},
\]
embedding dimension $12$, and multiplicity $24$. Their evaluation map has image equal to the maximal ideal. The different embedding dimensions rule out an isomorphism between these local rings.

\begin{table}[htbp]
\centering
\small
\begin{tabularx}{\textwidth}{@{}>{\raggedright\arraybackslash}Xcc@{}}
\toprule
 & Present example & \cite[Example~3.3]{CGI}\\
\midrule
Ground field & Any field $k$ & $\QQ$\\
Embedding dimension & $26$ & $12$\\
Hilbert--Samuel multiplicity & $56$ & $24$\\
Minimal generators of the ideal & $2$ & $2$\\
Colength of its trace ideal & $3$ & $1$\\
\bottomrule
\end{tabularx}
\caption{Numerical invariants of the two local examples.}\label{tab:comparison}
\end{table}

Santiba\~nez-Leal~\cite{FSL} reports computational Frobenius minimality and minimum-layer uniqueness for this example, within the class of two-generated monomial ideals over symmetric numerical semigroup rings, as well as a separate infinite family and further structural results. None of those minimality or classification results is needed here. Our theorem is an explicit existence statement and does not claim minimality among arbitrary one-dimensional Gorenstein domains or modules.

The complete-intersection setting is different. Garc\'ia-S\'anchez and Leamer~\cite{GL} prove the conjecture for two-generated monomial ideals over complete-intersection numerical semigroup rings. Thus our semigroup ring is not a complete intersection. The present construction disproves the general Gorenstein formulation, not its restriction to complete intersections. Likewise, vanishing of $\Ext_R^1(I,I)$ alone does not establish vanishing of every higher self-extension; no counterexample to an all-degrees self-extension conjecture is asserted.

\section{Exact verification and provenance}\label{sec:verification}
The source distribution includes two standalone Python programs in its \texttt{anc} directory. Both use only the standard library and integer arithmetic.

\pagebreak
The program \texttt{verify\_counterexample.py} starts from $G$ and computes the Ap\'ery vector of $\Gamma$ modulo $56$ by Dijkstra's algorithm. A relative $\Gamma$-ideal means a nonempty subset $H\subseteq\ZZ$ bounded below with $H+\Gamma\subseteq H$. Its relative Ap\'ery vector records the least member in each residue class modulo $56$. Membership is then exact in every degree:
\[
 n\in H\quad\Longleftrightarrow\quad n\geq h_{n\bmod56}.
\]
Intersections and unions correspond to coordinatewise maxima and minima. For a translate and a Minkowski sum, the corresponding formulas are
\[
 (h+s)_r=s+h_{r-s},\qquad
 (h\mathbin{+}h')_r=\min_{i\bmod56}(h_i+h'_{r-i}),
\]
where subscripts are reduced modulo $56$. Equality of such vectors therefore checks an infinite ideal identity, rather than membership up to a guessed cutoff.

The program verifies the generator minimality, the entire gap and membership lists, symmetry, the finite addition identities, the colon and inverse-ideal identities, and every support and numerical invariant in Section~\ref{sec:structure} and Corollary~\ref{cor:Tor}. It also compares its membership oracle with a separately computed dynamic-programming table.

The program \texttt{verify\_tensor\_graph.py} does not import the Ap\'ery implementation. It builds semigroup membership by dynamic programming and checks all tensor graphs below the proved bound. It additionally checks two small examples that do have torsion and a principal-ideal example, to test both outcomes of the graph calculation. Reproduce the checks from the source-package root with
\begin{verbatim}
python3 anc/verify_counterexample.py
python3 anc/verify_tensor_graph.py
\end{verbatim}
Deterministic expected outputs are supplied alongside the scripts. These calculations certify the explicit arithmetic; the algebraic identifications and the infinite tail bound are proved in the paper.

The original example was found in an AI-assisted search directed by the author. GPT-5.6 Pro assisted with the search, early drafting, and initial verification code. ChatGPT also assisted with the preparation and exact checks of this version. After the example was circulated, Craig Huneke independently recomputed its semigroup and colon-ideal data and verified the rigidity conclusion; his note is available in the companion repository~\cite{HunekeNote,PhamRepo}. The human author assumes responsibility for the statements of the paper.

The repository reference is pinned to an existing commit for the original verification materials. The verifiers for the present version are included directly with the arXiv source and do not require network access or any repository checkout.

\section*{Acknowledgments}
The author thanks Craig Huneke for independently verifying the example and preparing a separate verification note, and Srikanth B. Iyengar for helpful correspondence and for drawing attention to related work. The author also thanks the researchers whose work on rigid ideals and numerical semigroup rings supplies the algebraic framework used here.

\clearpage
\appendix
\section{Complete pair-sum certificate}\label{app:pairs}
The set $B$ has $12$ elements and $C$ has $49$ elements. The table below gives a decomposition of every element of $C$ as a sum of two members of $B$. Conversely, enumerating the $78$ unordered pairs with repetition from $B$ produces no value outside $C$.
\begin{center}
\small
\renewcommand{\arraystretch}{1.15}
\begin{tabular}{@{}lllll@{}}
$112=56+56$ & $113=56+57$ & $114=56+58$ & $115=57+58$ & $116=58+58$\\
$119=56+63$ & $120=56+64$ & $121=57+64$ & $122=58+64$ & $126=63+63$\\
$127=63+64$ & $128=64+64$ & $129=56+73$ & $130=57+73$ & $131=56+75$\\
$132=56+76$ & $133=57+76$ & $134=58+76$ & $135=56+79$ & $136=57+79$\\
$137=56+81$ & $138=56+82$ & $139=56+83$ & $140=57+83$ & $141=58+83$\\
$142=63+79$ & $143=64+79$ & $144=63+81$ & $145=63+82$ & $146=63+83$\\
$147=64+83$ & $148=73+75$ & $149=73+76$ & $150=75+75$ & $151=75+76$\\
$152=73+79$ & $154=73+81$ & $155=73+82$ & $156=73+83$ & $157=75+82$\\
$158=75+83$ & $159=76+83$ & $160=79+81$ & $161=79+82$ & $162=79+83$\\
$163=81+82$ & $164=81+83$ & $165=82+83$ & $166=83+83$ & \\
\end{tabular}
\end{center}

\end{document}